\documentclass[a4paper, twoside,12pt]{article}
\usepackage{fancyhdr}
\usepackage{latexsym}
\usepackage{amsmath}
\usepackage{graphicx}
\usepackage{amssymb}
\usepackage{stmaryrd}
\usepackage{slashbox}
\usepackage{eurosym}
\usepackage{multicol}
\usepackage{color}
\usepackage{lastpage}
\usepackage{ifthen}
\usepackage[amsmath,amsthm,thref,thmmarks,hyperref]{ntheorem}
\usepackage[colored]{shadethm}
\usepackage{hyperref}
\usepackage{pfnote}
\usepackage{fnpos}
 \usepackage[english]{babel}        
\usepackage[T1]{fontenc}          
\usepackage{graphicx}             
\usepackage{makeidx}
\usepackage{fancybox}
\usepackage{framed}
\usepackage{variations}
\usepackage{fancyhdr}
 {\markboth{\sl B. Amri }{\sl
 Riesz transform}

\setlength{\topmargin}{-0.3in}
\setlength{\topskip}{0.3in}
\setlength{\textheight}{9.5in}
\setlength{\textwidth}{6in}
\setlength{\oddsidemargin}{0.1in}
\setlength{\evensidemargin}{0.1in}

\newtheorem{thm}{Theorem}[section]

\newtheorem{lemma}[thm]{Lemma}
\newtheorem{prop}[thm]{Proposition}

\numberwithin{equation}{section}
\renewcommand{\thefootnote}

\newcommand{\set}[1]{\left\{#1\right\}}

\author{B\'{e}chir Amri}
\date{\small
University of Tunis, Preparatory Institut of Engineer
Studies of Tunis, Department of Mathematics,
1089 Montfleury Tunis, Tunisia\\ bechir.amri@ipeit.rnu.tn }
\title{Riesz transforms for Dunkl Hermite expansions}

\begin{document}

\maketitle
\begin{abstract}
In the present paper,  we establish that   Riesz transforms  for Dunkl Hermite expansions as introduced in \cite{NS} are
singular integral operators with H\"{o}rmander's  type conditions  and we show that they are  bounded on $L^p(\mathbb{R}^d, d\mu_\kappa)$,  $1<p<\infty$.
 \footnote{
\par Author  partially supported by DGRST project 04/UR/15-02 and CMCU program 10G 1503.}

\end{abstract}
\section{Introduction.}
In \cite{NS} the authors introduced the Riesz transforms  related to the Dunkl harmonic
oscillator  $L_\kappa$ and  they  proved that when the group of reflections is isomorphic to $\mathbb{Z}_2^d$ such operators are $L^p$ bounded
 with $ 1<p<\infty$.
The aim of this paper is to present
an extension of this result to general group of reflections  in arbitrary dimensions.
Our approach consists in the application of the standard theory of Calder�n-Zygmund operators.
The setting, which
is described in more details in Section 2, is as follows:
Let $R$ be a (reduced) root system on $\mathbb{R}^d$,  normalized so that $<\alpha,\alpha> =|\alpha|^2= 2$ for all $\alpha\in R$, where $\langle.,.\rangle$ denotes
 the usual Euclidean inner product and |.| its induced norm. Let $G$ be the associated reflection group and  $\kappa:\; R \rightarrow [0,+\infty[$
 be a $G$-invariant  function on $R$. The Dunkl operators
$T_j^\kappa$, $1\leq j\leq d$,  associated with $R$ and $\kappa$, which were introduced in \cite{Du}, are given by
\begin{eqnarray*}
T_j^\kappa(f)(x)=\frac{\partial f}{\partial x_j}(x)+\sum_{\alpha\in R_+}\kappa(\alpha)\alpha_j\frac{f(x)-f(\sigma_\alpha.x)}{\langle x, \alpha\rangle},\qquad x\in\mathbb{R}^d,\; f\in \mathcal{C}^1(\mathbb{R}^d)
 \end{eqnarray*}
where  $R_+$ denotes a positive subsystem of $R$ and $\sigma_\alpha$  the reflection in the hyperplane orthogonal to $\alpha$ ie:
\begin{eqnarray}\label{sigma}
\sigma_\alpha(x)=x-\langle x,\alpha\rangle\alpha, \qquad x\in\mathbb{R}^d.
\end{eqnarray}
The definition is of course independent of the choice of $R_+$ since $\kappa$ is G-invariant.
 We introduce the measure $d\mu_\kappa(x)=w_\kappa(x)dx$ where
\begin{eqnarray}\label{mu}
 w_\kappa(x)=\prod_{\alpha\in R_+}|\langle\alpha,x\rangle|^{2\kappa(\alpha)},\qquad x\in\mathbb{R}^d
\end{eqnarray}
which is $G$-invariant and homogenous of degree $2\gamma_\kappa$,
$$\gamma_\kappa=\sum_{\alpha\in R+}\kappa(\alpha).$$
\par The Dunkl harmonic oscillator is the operator  $L_\kappa= -\Delta_\kappa+|x|^2$ where $\Delta_\kappa $ denotes the Dunkl  Laplacian operator
$\displaystyle{ \Delta_\kappa=\sum_{j=1}^d (T_j^\kappa)^2}$. In particular,  the operator $ L_\kappa$ can be written as
 $\displaystyle{  L_\kappa=\frac{1}{2}\sum_{j=1}^d(\delta_j^\kappa \delta_j^{\kappa^*}+ \delta_j^{\kappa^*}\delta_j^\kappa)}$,
where
$\delta_j^\kappa=T_j^\kappa+x_j$ and $\delta_j^{\kappa^*}=-T_j^\kappa+x_j$. The Riesz transforms  related to the Dunkl harmonic
oscillator  are defined as  natural generalizations  of the classical ones ( see \cite{Th}) by
\begin{eqnarray*}
R_j^\kappa = \delta_j^\kappa L_\kappa^{-\frac{1}{2}}, \quad 1\leq j\leq d.
\end{eqnarray*}
The $L^2$ boundedness  of  these operators can be easily obtained  from the  Dunkl Hermite expansions.  Closely related to the integral operators
the  key new ingredient leading to $L^p$ boundedness   is the following:
\begin{thm}\label{th1}
 Let $T$ be a bounded operator on $L^2(\mathbb{R}^d,d\mu_\kappa)$   and $K$ be a measurable function on $
 \mathbb{R}^d\times \mathbb{R}^d\setminus\set{(x,g.x);\; x\in\mathbb{R}^d,\; g\in G}$  such that
 \begin{equation}\label{sing}
 T(f)(x)=\int_{\mathbb{R}^N}K(x,y)f(y)d\mu_\kappa(y),
 \end{equation}
    for all compactly supported $f$ in $L^2(\mathbb{R}^d,d\mu_\kappa)$ with
    $supp(f)\cap G.x =\emptyset$. If $K$ satisfies  H\"{o}rmander's  type conditions: there exists
a positive constant $C$  such that for all $\;y , y_0 \in \mathbb{R}^d,$
\begin{eqnarray}\label{ker} \int_{\min_{g\in G}|g.x-y|>2|y-y_0|}  |K(x,y)-K(x,y_0)|\;d\mu_\kappa(x) &\leq& C, \\
 \int_{\min_{g\in G}|g.x-y|>2|y-y_0|}  |K(y,x)-K( y_0,x)|\;d\mu_\kappa(x)&\leq& C,
\end{eqnarray}
      then $T$  extends to a bounded operator on $L^p(\mathbb{R}^d,d\mu_\kappa)$  for    $1<p< \infty$.
\end{thm}
This theorem is stated and proved in (\cite{BA}, \cite{BS}) and in \cite{AB} when $d=1$. We will show that  Riesz transform $R_j^\kappa$ has an integral representation satisfying  the  conditions in  Theorem \ref{th1} which proves the following main result.
\begin{thm}\label{th2}
The  Riesz transforms $R_j^\kappa$, $1\leq j\leq d$ can be extended to a bounded operators on $L^p(\mathbb{R}^d,d\mu_\kappa)$ to itself,
with  $1<p<\infty$.
\end{thm}
Finally, throughout this paper, C will denote a positive constant whose value is not necessarily the same at each occurrence.
 \section{Background and outline of the proof. }
We start by recalling    some  basic notations and   facts  from the  Dunkl  theory. For more details see references
  \cite{Du,Je,R1,R4,R2, R3,Xu}  and the literature cited there.
 \par First of all we note the following product rule, which is confirmed by a short calculation:
\begin{eqnarray}\label{id}
T_j^\kappa(\varphi\psi)=T_j^\kappa(\varphi)\psi+\varphi T_j^\kappa(\psi),\qquad 1\leq j\leq d,
\end{eqnarray}
if $\varphi,\psi\in C^1(\mathbb{R}^d)$ and  at least one of them is $G$-invariant.
\par The Dunkl kernel $E_\kappa$ is defined  on $\mathbb{R}^d\times\mathbb{C}^d$ by:   for  $y\in  \mathbb{C}^d$, $E_\kappa(.,y)$ is  the unique solution of the system:
\begin{eqnarray}\label{TE}
 T_j^\kappa f=y_jf, \quad f(0)=1.
\end{eqnarray}
 This kernel is symmetric with respect to its arguments and  has a unique holomorphic extension on $\mathbb{C}^d\times \mathbb{C}^d$. Moreover,
  $$E_\kappa(\lambda x,y)=E_\kappa(x,\lambda y)\quad\text{and}\quad E_\kappa(gx,gy)=E_\kappa(x,y)$$
for all $x,y\in\mathbb{C}^d$, $\lambda\in \mathbb{\mathbb{C}}$ and $g\in G$. The kernel $E_\kappa$ is connected with the exponential function by the Bochner-type representation
  \begin{eqnarray}\label{dk}
E_\kappa(x,y)=\int_{\mathbb{R}^d}e^{\langle \eta,y \rangle} d\nu_x(\eta), \qquad x,y \in\mathbb{R}^d
\end{eqnarray}
   where $\nu_x$ is a probability measure
supported in the the convex hull $co(G.x)$ and satisfies:
\begin{eqnarray}\label{mg}
\nu_{r x}(r B)=\nu_x(r^{-1}B)\quad \text{and}\quad \nu_{gx}( B)=\nu_x(g^{-1}B)
\end{eqnarray}
for each $r > 0$, $g \in G$ and each Borel set $B \subset \mathbb{R}^d$.
\par  The Dunkl transform    is defined,
for $f \in L^1(\mathbb{R}^d, d\mu_\kappa)$ by:
 \begin{eqnarray*}
 \mathcal{F}_\kappa(f)(\xi)=\frac{1}{c_\kappa}
\int_{\mathbb{R}^N}\!f(x)\,E_\kappa(-i\,\xi,x)d\mu_\kappa(x),\quad c_\kappa\,=\int_{\mathbb{R}^N}\!e^{-\frac{|x|^2}2}\,d\mu_\kappa(x).
  \end{eqnarray*}
   It plays the same
role as the Fourier transform in classical Fourier analysis $ (\kappa\equiv0)$ and enjoys
similar properties.
\par On  $L^2(\mathbb{R}^d,d\mu_\kappa)$ we define  the Dunkl translation operator
 $\tau_x,\; x\in\mathbb{R}^d$ by
\begin{eqnarray*}\label{dutr}
\mathcal{F}_\kappa(\tau_x^\kappa(f))(y)=
E_\kappa(ix,y)\mathcal{F}_\kappa (f)(y), \qquad y\in\mathbb{R}^d.
\end{eqnarray*}
When  $f$ is a  continuous radial function in $L^2(\mathbb{R}^d,d\mu_\kappa)$ with
$f(y)=\widetilde{f}(|y|)$, an explicit formula of $ \tau^\kappa_x(f)$ is given by
\begin{eqnarray}\label{trad}
              \tau^\kappa_x(f)(y)=\int_{\mathbb{R}^{N}}\widetilde{f}\Big(\;\sqrt{|x|^2+|y|^2+2<y,\eta>}\;\Big)d\nu_x(\eta).
\end{eqnarray}
This formula is  proved  by   M. R\"{o}sler \cite{R3}  for  Schwartz functions and     extended  to continuous functions by F. Dai and H.  Wang \cite {FW}. The Dunkl translation operator satisfies:
\begin{itemize}
\item [(i)]For all $x,y\in\mathbb{R}^d$ and $f\in L^2(\mathbb{R}^d,d\mu_\kappa)$,
\begin{eqnarray}\label{p}
\tau^\kappa_x(f)(y)=\tau^\kappa_y(f)(x).
\end{eqnarray}
  \item [(i)]For all $x\in\mathbb{R}^d$ and $f\in L^2(\mathbb{R}^d,d\mu_\kappa)\cap L^1(\mathbb{R}^d,d\mu_\kappa)$,
\begin{eqnarray}\label{int}
\int_{{\mathbb{R}}^d}\tau^\kappa_x(f)(y)d\mu_\kappa(y)=\int_{{\mathbb{R}}^d }f (y)d\mu_\kappa(y).
\end{eqnarray}
  \end{itemize}
     \par Let  $\mathcal{P} = \mathbb{C}[\mathbb{R}^d]$  the algebra of polynomial functions on $\mathbb{R}^d$ and $\mathcal{P}_N$, $N\in \mathbb{N}$ the subspace of homogenous polynomials of degree $N$. In \cite{CD},  C. F. Dunkl  introduced the following bilinear form on $\mathcal{P}$,

      \begin{eqnarray*}\label{lap}
        [p,q]_\kappa=\left(p(T)q\right)(0)=c_\kappa^{-1}\int_{\mathbb{R}^d}e^{\frac{-\Delta_\kappa}{2}}p(x)
        e^{\frac{-\Delta_\kappa}{2}}q(x)e^{\frac{-|x|^2}{2}}d\mu_\kappa(x).
\end{eqnarray*}
     where the operator $p(T)$ derived from $p(x)$ by replacing $x_j$ by $T_j^\kappa$.
    According to the  identity \cite{R1}
     $$\left(e^{\frac{-\triangle_\kappa}{2}}p\right)(\sqrt{2}x)=2^{\frac{N}{2}}\left(e^{\frac{-\triangle_\kappa}{4}}p\right)(x),\qquad p\in\mathcal{P}_N$$
    one has
 \begin{eqnarray*}
        [p,q]_\kappa=m_\kappa2^{N}\int_{\mathbb{R}^d}e^{\frac{-\Delta_\kappa}{4}}p(x)
        e^{\frac{-\Delta_\kappa}{4}}q(x)e^{-|x|^2}d\mu_\kappa(x), \qquad\; p,\;q\in\mathcal{P}_N
\end{eqnarray*}
where $m_\kappa=2^{\gamma_\kappa+\frac{d}{2}}c_\kappa^{-1}$. Then for a given  orthonormal basis $\set{\varphi_n,n\in \mathbb{N}^d}$ of $\mathcal{P}$ with respect to $[. , .]_\kappa$ such that $\varphi_n\in\mathcal{P}_{|n|}$ and with real coefficients we define the generalized
Hermite polynomials  $H_n^\kappa$ and the  generalized Hermite functions $h_n^\kappa$ on $\mathbb{R}^d$ by
  \begin{eqnarray*}
 H_n^\kappa(x)=2^{|n|}e^{\frac{-\Delta_\kappa}{4}}\varphi_n(x),\;\text{and }\;
 h_n^\kappa(x)=2^{-\frac{|n|}{2}}\sqrt{m_\kappa}\;e^{\frac{-|x|^2}{2}}H_n(x),\quad n\in \mathbb{N}^d
 \end{eqnarray*}where here $|n|=\sum_{j=1}^dn_j$.
  These are  described essentially in \cite{R1}.
The  important basic properties of $H_n^\kappa$ and $h_n^\kappa$  are \\
 \begin{itemize}
   \item [(i)]  Mehler-formula: For $r\in \mathbb{C}$ with $|r|<1$ and all $x,y\in \mathbb{R}^d$,
 \begin{eqnarray}\label{Meh}
 \sum_{n\in\mathbb{N}^d}\frac{H_n^\kappa(x)H_n^k(y)}{2^{|n|}}r^{|n|}=\frac{1}{(1-r^2)^{\gamma_\kappa+\frac{d}{2}}}e^{-\frac{r^2}{1-r^2}(|x|^2+|y|^2)}
E_\kappa(\frac{2r}{1-r^2}x,y).
  \end{eqnarray}
\item [(ii)] $h_n^\kappa$, $n\in \mathbb{N}^d$ are eigenfunctions of the operator $L_\kappa$,  with
$$L_\kappa(h_n^\kappa)=(2|n|+2\gamma_\kappa+d)h_n^\kappa.$$
    \item [(iii)]The set $\set{h_n^\kappa,\;n\in \mathbb{N}^d}$ forms an orthonormal basis of $L^2(\mathbb{R}^d; d\mu_\kappa)$.
     \end{itemize}
  \par Let  $ \mathcal{D}$ denotes the subspace of $L^2(\mathbb{R}^d; d\mu_\kappa)$ spanned by   $\set{h_n^\kappa,n\in \mathbb{N}^d}$. The Riesz transform $R_j^k$, $1\leq j\leq d$, associated with $L_\kappa$ is given on $ \mathcal{D}$ by
\begin{eqnarray}\label{def}
R_j^\kappa(f)=\delta_j^\kappa L_\kappa^{-\frac{1}{2}}(f)=\sum_{n\in\mathbb{N}^d}(2|n|+2\gamma_\kappa+d)^{-\frac{1}{2}}\langle f,h_n^\kappa\rangle \delta_j^\kappa h_n^\kappa.
\end{eqnarray}
Here $L_\kappa^{-\frac{1}{2}}$ is the continuous operator on $L^2(\mathbb{R}^d; d\mu_\kappa)$ given  via Spectral Theorem  by
$$L_\kappa^{-\frac{1}{2}}(f)=\sum_{n\in\mathbb{N}^d}(2|n|+2\gamma_\kappa+d)^{-\frac{1}{2}}\langle f,h_n^\kappa\rangle  h_n^\kappa.$$
\begin{prop}The Riesz transform $R_j^k$ extends  to a bounded operator from $L^2(\mathbb{R}^d; d\mu_\kappa)$ to itself.
\end{prop}
\begin{proof}
Let $\langle.,.\rangle_\kappa$ be the canonical inner product in $L^2(\mathbb{R}^d;d\mu_\kappa)$,
and $\|.\|_{2,\kappa}$ be the norm induced by $\langle.,.\rangle_\kappa$. Notice that  Dunkl operators are anti-symmetric with respect to the measure $\mu_\kappa$ \cite{Du}:
\begin{eqnarray}\label{Du}
\int_{\mathbb{R}^d}(T_j^\kappa f)(x)\varphi(x)d\mu_\kappa(x)=-\int_{\mathbb{R}^d}f(x)(T_j^\kappa \varphi)(x)d\mu_\kappa(x),
\end{eqnarray}
for all $C^1$-function $f$ and all Schwartz function $\varphi$. This implies that
$$\langle\delta_j^\kappa f,g\rangle_\kappa=\langle f,\delta_j^{\kappa^*} g\rangle_\kappa\quad \text{and}\quad
\langle\delta_j^{\kappa^*} f,g\rangle_\kappa=\langle f,\delta_j^{\kappa} g\rangle_\kappa$$
for all $f,g\in\mathcal{D}$. Then with the notation
$$R_j^{\kappa^*}(f)=\delta_j^{\kappa^*}L_\kappa ^{-\frac{1}{2}}(f)=\sum_{n\in\mathbb{N}^d}(2|n|+2\gamma_\kappa+d)^{-\frac{1}{2}}\langle f,h_n^\kappa\rangle \delta_j^{\kappa^*}h_n^\kappa ,\quad f\in\mathcal{D},$$
we have that
\begin{eqnarray*}
 \|R_j^\kappa f\|_{2,k}^2&\leq&  \|R_j^\kappa f\|_{2,k}^2+ \|R_j^{\kappa^*}f\|_{2,k}^2
 \\&=& \langle \delta_j^{\kappa^*}\delta_j L_\kappa^{-\frac{1}{2}}(f),L_\kappa^{-\frac{1}{2}}(f)\rangle_\kappa+
 \langle\delta_j\delta_j^{\kappa^*}L_\kappa^{-\frac{1}{2}}(f),  L_\kappa^{-\frac{1}{2}}(f)\rangle_\kappa\\
 &\leq&\sum_{j=1}^d\langle(\delta_j^{\kappa^*}\delta_j+\delta_j\delta_j^{\kappa^*})L_\kappa^{-\frac{1}{2}}(f),L_\kappa^{-\frac{1}{2}}(f)\rangle_\kappa\\
 &=&2\langle L_\kappa L_\kappa^{-\frac{1}{2}}(f),L_\kappa^{-\frac{1}{2}}(f)\rangle_\kappa=2\|f\|_{2,\kappa}^2.
 \end{eqnarray*}
  Since $\mathcal{D}$ is a dense subspace of  $L^2(\mathbb{R}^d; d\mu_\kappa)$, it follows by density argument that  $R_j^\kappa$ is uniquely extended to a  bounded  operator on
 $L^2(\mathbb{R}^d; d\mu_\kappa)$.
\end{proof}

\subsection{Proof of Theorem \ref{th2}}
 The  Hermite semigroup $e^{-tL_\kappa}$ ( $t\geq0$ ), is given on $ L^2(\mathbb{R}^d; d\mu_\kappa)$ by
$$e^{-tL_\kappa}(f)=\sum_{n\in\mathbb{N}^d}e^{-t(2|n|+2\gamma_\kappa+d)}\langle f,h_n^\kappa\rangle h_n^\kappa$$
and has  the following   integral representation
 \begin{eqnarray}\label{L_k}
e^{-tL_\kappa}(f)(x)=\int_{\mathbb{R}^d}k_t(x,y)f(y)d\mu_\kappa(y),\quad x\in \mathbb{R}^d
 \end{eqnarray}
  where from  Mehler-formula (\ref{Meh})
\begin{eqnarray}\label{eE}
 k_t(x,y)&=&\sum_{n\in\mathbb{N}^d}e^{-t(2|n|+2\gamma_\kappa+d)}h_n^\kappa(x)h_n^\kappa(y)\nonumber\\&=&
m_\kappa e^{-\frac{1}{2}(|x|^2+|y|^2)}e^{-t(2\gamma_\kappa+d)}\sum_{n\in\mathbb{N}^d}e^{-2t|n|}\frac{H_n^\kappa(x)H_n^\kappa(y)}{2^{|n|}}
\nonumber\\&=&m_\kappa(\sinh(2t))^{-\gamma_\kappa-\frac{d}{2}}e^{-\frac{\coth(2t)}{2}(|x|^2+|y|^2)}E_\kappa(\frac{x}{\sinh(2t)},y).
  \end{eqnarray}
According to  (\ref{dk}) it follows that
\begin{eqnarray*}
k_t(x,y)&=&m_\kappa(\sinh(t))^{-\gamma_\kappa-\frac{d}{2}}e^{-\frac{\coth(2t)}{2}(|x|^2+|y|^2)}
\int_{\mathbb{R}^d}e^{\frac{1}{\sinh(2t)}\langle y,\eta\rangle}d\nu_x(\eta)\\
&=&m_\kappa(\sinh(2t))^{-\gamma_\kappa-\frac{d}{2}}\int_{\mathbb{R}^d}e^{-\frac{\coth(2t)}{2}(|x|^2+|y|^2-2\langle y,\eta\rangle)}e^{-\tanh (t)\langle y,\eta\rangle}
 d\nu_x(\eta)
 \\&=& m_\kappa(\sinh(2t))^{-\gamma_\kappa-\frac{d}{2}}\int_{\mathbb{R}^d}e^{-\frac{\coth(2t)}{2}|y-\eta|^2-\tanh (t)\langle y,\eta\rangle}e^{-\frac{\coth(2t)}{2}(|x|^2-|\eta|^2)} d\nu_x(\eta).
\end{eqnarray*}
 Let $ k_t^0$ the  kernel of the classical Hermite semigroup  (\cite{Stemp}, \cite{Th}), which corresponds to $\kappa\equiv0$   and given by
\begin{eqnarray}\label{ks}\nonumber
 k_t^0(x,y)&=&(2\pi\sinh(2t))^{-\frac{d}{2}}e^{-\frac{\coth(2t)}{2}|x-y|^2-\tanh (t)\langle x,y\rangle}\\
&=&(2\pi\sinh(2t))^{-\frac{d}{2}}e^{-\frac{1}{4}(\tanh (t)|x+y|^2+\coth(t)|x-y|^2)}.
\end{eqnarray}
 Thus we can write
  \begin{eqnarray}\label{her}
 k_t(x,y)=m_\kappa(2\pi)^{-\frac{d}{2}}(\sinh2t)^{-\gamma_\kappa}\int_{\mathbb{R}^d} k_t^0( \eta,y)e^{-\frac{\coth(2t)}{2}(|x|^2-|\eta|^2)} d\nu_x(\eta).
 \end{eqnarray}
  We need   some  estimates on the kernel   $k_t$. As we will see later
 these estimates follow from the basic estimates on the kernel $k_t^0$ stated in the following lemma.
     \begin{lemma}\label{01}
There exist $C > 0$ and $a > 0$ such that for all $x, y \in \mathbb{R}^d$ and $1\leq i,j\leq d$   the following hold.
\\ 1)  For  all $0<t\leq1$
\begin{itemize}
  \item [(i)] $\displaystyle{\quad
       \Big| k_t^0(x,y)\Big|  \leq  Ct^{ -\frac{d}{2}}e^{-\frac{a}{ t}|x-y|^2 }  },$
       \item [(ii)]$\displaystyle{\quad
       \Big|y_j k_t^0(x,y)\Big|  \leq  Ct^{ -\frac{d+1}{2}}e^{-\frac{a}{ t}|x-y|^2 }  },$
  \item [(iii )]   $\displaystyle{\quad\Big|\frac{\partial k_t^0}{\partial y_j}(x,y)\Big|    \leq  Ct^{ -\frac{d+1}{2}}e^{-\frac{a}{ t}|x-y|^2 }},$
      \item [(iv)] $\displaystyle{ \quad\Big|y_j\frac{\partial k_t^0}{\partial y_i}(x,y)\Big|    \leq  Ct^{ -\frac{d}{2}-1}e^{-\frac{a}{ t}|x-y|^2 } }.$
\end{itemize}
2)  For  all $t >1$,
     \begin{itemize}
   \item[(v)] $\displaystyle{\quad
    \Big|k_t^0(x,y)\Big| \leq  Ce^{-dt}e^{ -a|x-y|^2 }},$\\
     \item  [(vi)]$\displaystyle{\quad\Big|y_jk_t^0(x,y)\Big|  \leq  Ce^{-dt}e^{ -a|x-y|^2 }}.$
     \end{itemize}
        \end{lemma}
\begin{proof}
 1) For $0<t\leq1$, $\sinh(t)$ and $\tanh(t)$ behave like $t$ and $\coth(t)$ behaves like $t^{-1}$ with the  following fact
\begin{eqnarray}\label{coth}
\coth(t)\geq \dfrac{1}{t},\qquad 0<t\leq1.
\end{eqnarray}
 The estimate $(i)$ follows from (\ref{ks}) and (\ref{coth}), since
$$ \Big|k_t^0(x,y)\Big| \leq (2\pi\sinh(2t))^{-\frac{d}{2}}e^{-\frac{1}{4}\coth(t)|x-y|^2)}\leq  Ct^{ -\frac{d}{2}}e^{-\frac{1}{ 4t}|x-y|^2 }. $$
As $|y_j|\leq |y_j-x_j|+|y_j+x_j|$ we have
\begin{eqnarray}\label{k1}\nonumber
\Big|(y_j-x_j)k_t^0(x,y)\Big|&\leq&C t^{-\frac{d}{2}}|y-x|e^{-\frac{1}{4}\coth(t)|x-y|^2}\\&\leq& C t^{-\frac{d}{2}}e^{-\frac{1}{8t}|x-y|^2}.
\end{eqnarray}
Similarly,
\begin{eqnarray}\label{k2}\nonumber
\Big|(y_j+x_j)k_t^0(x,y)\Big|&\leq &C t^{-\frac{d}{2}}\Big(|y+x|e^{-\frac{1}{4}\tanh(t)|x+y|^2}\Big)e^{-\frac{1}{4}\coth(t)|x-y|^2}\\
&\leq &C t^{-\frac{d+1}{2}} e^{-\frac{1}{4}\coth(t)|x-y|^2} \leq  C t^{-\frac{d+1}{2}} e^{-\frac{1}{4t}|x-y|^2}.
\end{eqnarray}
From this estimate (ii) follows immediately. We can also obtain  $(iii)$ from (\ref{k1}) and (\ref{k2}), since we have
\begin{eqnarray}\label{part}
\frac{\partial k_t^0}{\partial y_j}(x,y)=-\frac{1}{2}\Big(\tanh(t)|y_j+x_j|+\coth(t)|y_j-x_j|\Big)k_t^0(x,y).
\end{eqnarray}
The  estimate $(iv)$ follows from (\ref{part}) and the following:
\begin{eqnarray*}
\Big|(y_j-x_j)^2k_t^0(x,y)\Big|&\leq&C t^{-\frac{d}{2}}|y-x|^2e^{-\frac{1}{4}\coth(t)|x-y|^2}\\&\leq& C t^{-\frac{d}{2}}e^{-\frac{1}{8t}|x-y|^2},
\end{eqnarray*}
\begin{eqnarray*}
\Big|(y_j+x_j)^2k_t^0(x,y)\Big|&\leq &C t^{-\frac{d}{2}}\Big(|y+x|^2e^{-\frac{1}{4}\tanh(t)|x+y|^2}\Big)e^{-\frac{1}{4}\coth(t)|x-y|^2}\\
&\leq &C t^{-\frac{d}{2}-1} e^{-\frac{1}{4}\coth(t)|x-y|^2} \leq  C t^{-\frac{d}{2}-1} e^{-\frac{1}{4t}|x-y|^2},
\end{eqnarray*}
and
\begin{eqnarray*}
\Big|(y_j+x_j)(y_j-x_j)k_t^0(x,y)\Big|&\leq &C t^{-\frac{d}{2}}\Big(|y+x|e^{-\frac{1}{4}\tanh(t)|x+y|^2}\Big)\Big(|y_j-x_j|e^{-\frac{1}{4}\coth(t)|x-y|^2}\Big)\\
&\leq &C t^{-\frac{d}{2}-\frac{1}{2}} e^{-\frac{1}{8}\coth(t)|x-y|^2} \leq  C t^{-\frac{d}{2}-\frac{1}{2}} e^{-\frac{1}{8t}|x-y|^2}.
\end{eqnarray*}
Note that all estimates can be made with $a=\dfrac{1}{8}$.
\par 2) For $t\geq 1$ the proof is very similar to that of 1). We have to use the fact that both $\sinh(2t)$ and $\coth(2t)$ behave like $e^{2t}$
with $\coth(t)\geq 1$.
\end{proof}
 As a consequence of Lemma \ref{01}  we obtain the following:
\begin{lemma}\label{lem1}
 There exist $C > 0$ and $a > 0$ such that for all $x, y \in \mathbb{R}^d$ and $1\leq i,j\leq d$    the following hold.
\\ 1)  For  all  $0<t\leq1$,
   \begin{itemize}
    \item [(i)] $\displaystyle{\quad
       \Big| k_t (x,y)\Big|  \leq Ct^{-\gamma_\kappa-\frac{d}{2}}\;\tau^\kappa_x(e^{-\frac{b}{ t}|.|^2})(-y)  },$
       \item [(ii)] $\displaystyle{\quad
       \Big|y_j k_t (x,y)\Big|  \leq Ct^{-\gamma_\kappa-\frac{d+1}{2}}\;\tau^\kappa_x(e^{-\frac{b}{ t}|.|^2})(-y)  },$
    \item [(iii)]$\displaystyle{\quad\Big|\frac{\partial k_t}{\partial y_j} (x,y)\Big|    \leq  Ct^{-\gamma_\kappa-\frac{d+1}{2}}\;\tau^\kappa_x(e^{-\frac{b}{ t}|.|^2})(-y)},$
    \item [(iv)]$\displaystyle{\quad \Big|y_j\frac{\partial k_t }{\partial y_i}(x,y)\Big|    \leq  Ct^{-\gamma_\kappa-\frac{d}{2}-1}\;\tau^\kappa_x(e^{-\frac{b}{ t}|.|^2})(-y) }$.
  \end{itemize}
  2)For all  $ t>1$,
  \begin{itemize}
    \item [(v)]$\displaystyle{\Big| k_t (x,y)\Big|  \leq  Ce^{-(2\gamma_\kappa+d)t}\tau^\kappa_x(e^{ -b |.|^2})(-y)  },$
    \item [(vi)]$\displaystyle{\Big|y_j k_t (x,y)\Big|  \leq   Ce^{-(2\gamma_\kappa+d)t}\tau^\kappa_x(e^{ -b |.|^2})(-y) }.$
  \end{itemize}
     \end{lemma}

    \begin{proof}
From (\ref{her}), Lemma \ref{01} and (\ref{trad}) we have with  $b=\min(a,\frac{1}{4})$,
 \begin{eqnarray*}
  \Big| k_t (x,y)\Big|&\leq& Ct^{-\gamma_\kappa-\frac{d}{2}}\int_{\mathbb{R}^d} e^{-\frac{a}{ t}|y-\eta|^2}e^{-\frac{ 1}{4t}(|x|^2-|\eta|^2)} d\nu_x(\eta)\\
 &\leq& Ct^{-\gamma_\kappa-\frac{d}{2}}\int_{\mathbb{R}^d} e^{-\frac{b}{t}\;( | \eta-y|^2+|x|^2-|\eta|^2)}d\nu_x(\eta)
 \\&=& Ct^{-\gamma_\kappa-\frac{d}{2}}\int_{\mathbb{R}^d} e^{-\frac{b}{t}\;( |x|^2+|y|^2-2\langle y,\eta\rangle)}d\nu_x(\eta)
=Ct^{-\gamma_\kappa-\frac{d}{2}}
 \tau^\kappa_x(e^{\frac{-b}{ t}|.|^2})(-y)
 \end{eqnarray*}
 which proves (i). We obtain $(ii)-(vi)$  by similar way.
      \end{proof}
    \begin{lemma}\label{lem3}
There exist constants  $C>0$ and $c>0$ such that for all  $x,y\in\mathbb{R}^d$,   $ 0<t\leq1$ and $1\leq i, j\leq d$,
    \begin{itemize}
      \item [(i)]  $\displaystyle{\Big|(x_j-y_j) k_t (x,y)\Big|  \leq   Ct^{-\gamma_\kappa-\frac{d}{2}+\frac{1}{2}}\sum_{\alpha\in \mathcal{R}_+}\tau^\kappa_{\sigma_\alpha.x}(e^{-\frac{c}{ t}|.|^2})(-y)      },$
       \item[(ii)] $\displaystyle{\Big|(x_j-y_j) \frac{\partial k_t}{\partial y_i} (x,y)\Big|  \leq    Ct^{-\gamma_\kappa-\frac{d }{2}}\sum_{\alpha\in \mathcal{R}_+}\tau^\kappa_{\sigma_\alpha.x}(e^{-\frac{c}{ t}|.|^2})(-y)},$
      \end{itemize}
where $\mathcal{R}_+=R_+\cup\{0\}$ and with $\sigma_0=id$.
    \end{lemma}
    \begin{proof}
    In view of $(i)$ and $(iii)$ of Lemma \ref{lem1} it is enough to show that for some constant $c>0$,
    \begin{eqnarray}\label{tau}
     |y_j-x_j|\tau^\kappa_x(e^{-\frac{b}{ t}|.|^2})(-y)\leq Ct^{\frac{1}{2}}\sum_{\alpha\in R_+}\tau^\kappa_{\sigma_\alpha.x}(e^{-\frac{c}{ t}|.|^2})(-y).
    \end{eqnarray}
Making use of  (\ref{trad}), (\ref{dk}), (\ref{id})  and (\ref{TE})  we see that
     \begin{eqnarray}\label{aa}\nonumber
    (y_j-x_j)\tau^\kappa_x(e^{-\frac{b}{ t}|.|^2})(-y)&=& (y_j-x_j)e^{-\frac{b}{ t}(|x|^2+|y|^2)}E_\kappa(x,\frac{2b}{t}y)\\&=&
-\frac{t}{2b} T_j^\kappa\tau^\kappa_x(e^{-\frac{b}{ t}|.|^2})(-y).
\end{eqnarray}
However,
\begin{eqnarray*}
T_j^\kappa\tau^\kappa_x(e^{-\frac{b}{ t}|.|^2})(-y)&=&-\frac{2b}{t}\int_{\mathbb{R}^d}(y_j-\eta_j)e^{-\frac{b}{t}(|y-\eta|^2+|x|^2-|\eta|^2)}d\nu_x(\eta)\\&+&
\sum_{\alpha\in R_+}\kappa(\alpha)\alpha_j\int_{\mathbb{R}^d}\frac{e^{-\frac{b}{t}|y-\eta|^2}-e^{-\frac{b}{t}|\sigma_\alpha.y-\eta|^2}}{\langle y,\alpha\rangle}e^{-\frac{b}{t}(|x|^2-|\eta|^2)}d\nu_x(\eta)\\&=& I_1(x,y)+I_2(x,y).
\end{eqnarray*}
Then  we are going to estimate $I_1(x,y)$ and $I_2(x,y)$.
 It follows first that
 \begin{eqnarray}\label{I1}\nonumber
|I_1(x,y)|&\leq& \frac{2b}{t}\int_{\mathbb{R}^d}|y-\eta|e^{-\frac{b}{t}(|y-\eta|^2+|x|^2-|\eta|^2)}d\nu_x(\eta)\\&\leq&\nonumber
C t^{-\frac{1}{2}}\int_{\mathbb{R}^d}e^{-\frac{b}{2t}(|y-\eta|^2+|x|^2-|\eta|^2)}d\nu_x(\eta)\\&=&
C t^{-\frac{1}{2}}\tau^\kappa_x(e^{\frac{-b}{ 2t}|.|^2})(-y).
\end{eqnarray}
To establish a similar  estimate for $I_2(x,y)$ we need to estimate the function
$$\phi_t(u,v)=\frac{1-e^{-\frac{2b}{t}uv}}{v}\qquad u,v\in\mathbb{R}, \; uv\geq 0,$$
since  in view of (\ref{sigma}) we can write
\begin{eqnarray}\label{ye}
\frac{e^{-\frac{b}{t}|y-\eta|^2}-e^{-\frac{b}{t}|\sigma_\alpha.y-\eta|^2}}{\langle y,\alpha\rangle}=
e^{-\frac{b}{t}|y-\eta|^2}\left\{\frac{1-e^{-\frac{2b}{t}\langle\eta,\alpha\rangle\langle y,\alpha\rangle}}{\langle y,\alpha\rangle}\right\}\nonumber\\=
e^{-\frac{b}{t}|y-\eta|^2}\phi_t\Big(\langle\eta,\alpha\rangle,\langle y,\alpha\rangle\Big).
\end{eqnarray}
We proceed as follows: if $|u|\leq 2|v|$ then
$$|\phi_t(u,v)|= \frac{1-e^{-\frac{2b}{t}|u||v|}}{|v|}\leq\frac{1-e^{-\frac{4b}{t}v^2}}{|v|}\leq Ct^{-\frac{1}{2}}$$
and when $|u|> 2|v|$
$$|\phi_t(u,v)|= \frac{1-e^{-\frac{2b}{t}|u||v|}}{|v|}
\leq Ct^{-1} |u|\leq 2Ct^{-1}(|u|-|v|)=2Ct^{-1}|u-v|. $$
 Hence, we obtain
\begin{eqnarray}\label{phi}
|\phi_t(u,v)|\leq C(t^{-\frac{1}{2}}+t^{-1}|u-v|),\qquad uv\geq0.
\end{eqnarray}
It follows from  (\ref{ye}) and  (\ref{phi}) that when $\langle\eta,\alpha\rangle\langle y,\alpha\rangle\geq 0$,
\begin{eqnarray}\label{1}\nonumber
\left|\frac{e^{-\frac{b}{t}|y-\eta|^2}-e^{-\frac{b}{t}|\sigma_\alpha.y-\eta|^2}}{\langle y,\alpha\rangle}\right|
&\leq &C\Big(t^{-\frac{1}{2}}+t^{-1}|\langle y-\eta,\alpha\rangle|\Big)e^{-\frac{b}{t}|y-\eta|^2}\\&\leq&
C\Big(t^{-\frac{1}{2}}+\sqrt{2}t^{-1}| y-\eta|\Big)e^{-\frac{b}{t}|y-\eta|^2}\nonumber\\&\leq&
Ct^{-\frac{1}{2}}e^{-\frac{b}{2t}|y-\eta|^2}.
\end{eqnarray}
However,  when $\langle\eta,\alpha\rangle\langle y,\alpha\rangle\leq0 $  we have that  $\langle\eta,\alpha\rangle\langle \sigma_\alpha.y,\alpha\rangle\geq 0$, since
 $$\langle \sigma_\alpha.y,\alpha\rangle=\langle y, \sigma_\alpha.\alpha\rangle=-\langle y,\alpha\rangle.$$ In addition,
 (\ref{ye})  is invariant under replacing  $y$ by $\sigma_\alpha.y$, hence  we can write
\begin{eqnarray*}
\frac{e^{-\frac{b}{t}|y-\eta|^2}-e^{-\frac{b}{t}|\sigma_\alpha.y-\eta|^2}}{\langle y,\alpha\rangle}&=&
e^{-\frac{b}{t}|\sigma_\alpha.y-\eta|^2}\phi_t\Big(\langle\eta,\alpha\rangle,\langle \sigma_\alpha.y,\alpha\rangle\Big)
\end{eqnarray*}
and we also have
\begin{eqnarray}\label{2}
\left|\frac{e^{-\frac{b}{t}|y-\eta|^2}-e^{-\frac{b}{t}|\sigma_\alpha.y-\eta|^2}}{\langle y,\alpha\rangle}\right|
\leq  Ct^{-\frac{1}{2}}e^{-\frac{b}{2t}|\sigma_\alpha.y-\eta|^2}.
\end{eqnarray}
Combining (\ref{1}) and (\ref{2}) yields
\begin{eqnarray*}
\left|\frac{e^{-\frac{b}{t}|y-\eta|^2}-e^{-\frac{b}{t}|\sigma_\alpha.y-\eta|^2}}{\langle y,\alpha\rangle}\right|
\leq  Ct^{-\frac{1}{2}}(e^{-\frac{b}{2t}|y-\eta|^2}+e^{-\frac{b}{2t}|\sigma_\alpha.y-\eta|^2}).
\end{eqnarray*}
Therefore,
$$\int_{\mathbb{R}^d}\left|\frac{e^{-\frac{b}{t}|y-\eta|^2}-e^{-\frac{b}{t}|\sigma_\alpha.y-\eta|^2}}{\langle y,\alpha\rangle}\right|e^{-\frac{b}{t}(|x|^2-|\eta|^2)}d\nu_x(\eta)\qquad\qquad\qquad\qquad\qquad\qquad\qquad$$
\begin{eqnarray*}
&\leq&Ct^{-\frac{1}{2}}\left(
\int_{\mathbb{R}^d}e^{-\frac{b}{2t}(|y-\eta|^2+|x|^2-|\eta|^2)}d\nu_x(\eta)
+\int_{\mathbb{R}^d}e^{-\frac{b}{2t}(|\sigma_\alpha.y-\eta|^2+|x|^2-|\eta|^2)}d\nu_x(\eta)\right)\\&=&Ct^{-\frac{1}{2}}\Big(\tau^\kappa_xe^{-\frac{b}{2 t}|.|^2})(-y)+\tau^\kappa_{x}(e^{-\frac{b}{2 t}|.|^2})(-\sigma_\alpha.y)\Big)
\\&=&Ct^{-\frac{1}{2}}\Big(\tau^\kappa_{x}(e^{-\frac{b}{2 t}|.|^2})(-y)+\tau^\kappa_{\sigma_\alpha.x}(e^{-\frac{b}{2 t}|.|^2})(-y)\Big)
\end{eqnarray*}
where the second equality follows be a simple change of variable and (\ref{mg}).
 Thus we obtain
\begin{eqnarray}\label{I2}
|I_2(x,y)|\leq Ct^{-\frac{1}{2}}\sum_{\alpha\in \mathcal{R}_+}\tau^\kappa_{\sigma_\alpha.x}(e^{-\frac{b}{2 t}|.|^2})(-y).
\end{eqnarray}
Now, in view of (\ref{aa}) the estimates (\ref{I2}) and (\ref{I1})yield  (\ref{tau}) with $c=\frac{b}{2 }$.
\end{proof}
 In what follow we put
\begin{eqnarray}
\nonumber K_j(x,y)&=&\frac{1}{\sqrt{\pi}}\int_0^{+\infty}\delta_j^\kappa k_t(x,y)\frac{dt}{\sqrt{t}}\\&=&\label{k_j}
 \frac{1}{\sqrt{\pi}}\int_0^{+\infty}k_t(x,y)\Big((1-\coth (2t) )x_j+ \frac{1}{\sinh (2t) }y_j\Big)\frac{dt}{\sqrt{t}}.
  \end{eqnarray}
where $\delta^\kappa_j$ is taken with respect to the variable $x$. The last equality follows from  (\ref{eE}), (\ref{id} and  (\ref{TE}).
\begin{lemma}\label{lem4}
 For all $x,y\in \mathbb{R}^{d}$, $y\notin G.x$, the integral (\ref{k_j}) converge absolutely and
 \begin{eqnarray}\label{min}
|K_j(x,y)|\leq  C\min_{g\in G}|y-g.x|^{-2\gamma_\kappa-d}
 \end{eqnarray}
  \end{lemma}
 \begin{proof}
Let us observe  that for $x,y\in\mathbb{R}^d$ and  $\eta\in co(G.x)$ (see \cite{R3}, Th. 5.1)
\begin{eqnarray}\label{sq}
\min_{g\in G}|y-g.x|^2\leq|x|^2+|y|^2-2\langle y,\eta\rangle\leq \max_{g\in G}|y-g.x|^2.
\end{eqnarray}
Using  Lemma  \ref{lem3}-$(i)$, Lemma \ref{lem1}-$(iii)$,  (\ref{trad}) and (\ref{sq}) we have for some constant $c>0$ \\\\
 $\displaystyle{\int_0^{1}\Big|k_t(x,y)\Big((1-\coth(2t))x_j+ \frac{1}{\sinh(2t)}y_j\Big)\Big|\frac{dt}{\sqrt{t}}
}$
 \begin{eqnarray}\label{m}\nonumber
 &=& \int_0^{1}\Big|k_t(x,y)\Big((1-\coth(2t))(x_j-y_j) +(1-\tanh (t)) y_j\Big)\Big|\frac{dt}{\sqrt{t}}
  \\\nonumber&\leq&
 C\;\sum_{\alpha\in \mathcal{R}_+}\int_0^{1}   t^{-\gamma_\kappa-\frac{d+2}{2}}\;\tau^\kappa_{\sigma_\alpha.x}(e^{-\frac{c}{ t}|.|^2})(-y) dt  \\
 &=&\nonumber C\;\sum_{\alpha\in \mathcal{R}_+}\int _{\mathbb{R}^d}\left(\int_0^{1}t^{-\gamma_\kappa-\frac{d+2}{2}}e^{-\frac{c}{t}(|x|^2+|y|^2-2\langle y,\eta\rangle)}dt\right)d\nu_{\sigma_\alpha.x}(\eta)
 \\&\leq& \nonumber C \;\int_0^{ +\infty }   u^{-\gamma_\kappa-\frac{d+2}{2} }\;  e^{-\frac{c}{ u}}   du \;\sum_{\alpha\in \mathcal{R}_+}\int _{\mathbb{R}^d}(|x|^2+|y|^2-2\langle y,\eta\rangle)^{  -\gamma_\kappa-\frac{d}{2} } d\nu_{\sigma_\alpha.x}(\eta)
 \\ &\leq & C\;\min_{g\in G}|y-g.x|^{-2\gamma_\kappa-d}\;.
 \end{eqnarray}
In addition,  from Lemma \ref{lem1}-$(iii)$ it follows  that\\\\
$\displaystyle{\int_1^{+\infty}\Big|k_t(x,y)\Big((1-\coth2t)x_j+ \frac{1}{\sinh2t}y_j\Big)\Big|\frac{dt}{\sqrt{t}}}$
\begin{eqnarray}\label{n}
\leq C \tau^\kappa_x(e^{ -c |.|^2})(-y)  \leq C e^{ -c \min_{g\in G}|y-g.x|^2}\leq C\;\min_{g\in G}|y-g.x|^{-2\gamma_\kappa-d}\;.
\end{eqnarray}
Note here that as the kernel $k_t$ is symmetric we have used
$$|x_j k_t(x,y)|= |x_jk_t(y,x)|\leq  C \tau^\kappa_y(e^{ -c |.|^2})(-x)= C \tau^\kappa_x(e^{ -c |.|^2})(-y),$$
since
$$\tau^\kappa_x(e^{ -c |.|^2})(-y)=e^{-c(|x|^2+|y|^2)}E_\kappa(2cy,x)$$
which is also  symmetric with respect to variables $x $ and $y$.
We immediately get (\ref{min}) from (\ref{m}) and  (\ref{n}).
\end{proof}
\begin{prop}\label{prop1}
The Riesz transform $R_j^\kappa$ satisfies
\begin{eqnarray}\label{io}
R_j^\kappa(f)(x)=\int_{\mathbb{R}^d}K_j(x,y)f(y)d\mu_\kappa(y)
\end{eqnarray}
for all compactly supported function $f\in L^2(\mathbb{R}^d; d\mu_\kappa)$ with  $G.x \cap supp(f)=\emptyset$.
\end{prop}
\begin{proof}
According to Lemma \ref{lem4} the integral (\ref{io}) converges absolutely. In order, to establish  (\ref{io})
we first  write  $L_\kappa^{-\frac{1}{2}}$ in the following way
 \begin{eqnarray*}
L_\kappa^{-\frac{1}{2}}(f)(x)&=&\frac{1}{\sqrt{\pi}}\int_0^{+\infty}e^{-tL_\kappa}(f)(x)\frac{dt}{\sqrt{t}}.
\end{eqnarray*}
Then from   (\ref{L_k}) one has
\begin{eqnarray}\label{L}
L_\kappa^{-\frac{1}{2}}(f)(x)
 = \frac{1}{\sqrt{\pi}}\int_0^{+\infty}\int_{\mathbb{R}^d}k_t(x,y)f(y)d\mu_\kappa(y)\frac{dt}{\sqrt{t}}.
\end{eqnarray}
As in the proof of Lemma \ref{lem4} this  integral converge absolutely
when  $G.x\cap supp(f)=\emptyset$. Indeed, by   $(i)$ and $(v)$ of Lemma
\ref{lem1},\\\\
$\displaystyle{
 \int_0^{1}\int_{\mathbb{R}^d}\Big|k_t(x,y)f(y)\Big|d\mu_\kappa(y)\frac{dt}{\sqrt{t}}}$
\begin{eqnarray*}
&\leq &
 C \int_{\mathbb{R}^d}\int_0^{1}t^{-\gamma_\kappa-\frac{d}{2}}\;\tau^\kappa_x(e^{-\frac{b}{ t}|.|^2})(-y)  |f(y)|\frac{dt}{\sqrt{t}}d\mu_\kappa(y)
 \\&=&C \int_{\mathbb{R}^d}\int_{\mathbb{R}^d}\int_0^{1}t^{-\gamma_\kappa-\frac{d}{2}} e^{-\frac{b}{ t}(|x|^2+|y|^2-2\langle y,\eta\rangle)} |f(y)|\frac{dt}{\sqrt{t}}d\nu_x(\eta)d\mu_\kappa(y)\\
  \\&\leq&C \int_{\mathbb{R}^d}  \frac{|f(y)|}{\min_{g\in G}|y-g.x|^{2\gamma_\kappa+d-1}}  d\mu_\kappa(y)
  \end{eqnarray*}
and
\\\\
$\displaystyle{\int_1^{+\infty}\int_{\mathbb{R}^d}\Big|k_t(x,y)f(y)\Big|d\mu_\kappa(y)\frac{dt}{\sqrt{t}}}$
\begin{eqnarray*}
  &\leq &C \int_{\mathbb{R}^d}\int_1^{+\infty}e^{(-2\gamma_\kappa-d)t}\;\tau^\kappa_x(e^{ -b |.|^2 })(-y) |f(y)|\frac{dt}{\sqrt{t}}d\mu_\kappa(y)
 \\&=&C \int_{\mathbb{R}^d}\int_{\mathbb{R}^d}\int_1^{+\infty}e^{(-2\gamma_\kappa-d)t} e^{-b (|x|^2+|y|^2-2\langle y,\eta\rangle)} |f(y)|\frac{dt}{\sqrt{t}}d\nu_y(\eta)d\mu_\kappa(y)\\
 &\leq&  \int_{\mathbb{R}^d} \int_1^{+\infty}e^{(-2\gamma_\kappa-d)t}  |f(y)|\frac{dt}{\sqrt{t}}d\nu_y(\eta)d\mu_\kappa(y)\\&\leq&C \int_{\mathbb{R}^d}  |f(y)|d\mu_\kappa(y).
  \end{eqnarray*}
In the next, we will show that we can differentiate the integral (\ref{L}). Making use of  Lemma \ref{lem1}-$(iii) $ and (\ref{sq}), we deduce from the symmetry of the kernel $k_t$ that
$$\Big|\frac{\partial k_t}{\partial x_j} (x,y)\Big|=\Big|\frac{\partial k_t }{\partial y_j} (y,x)\Big|\leq
\left\{
  \begin{array}{ll}
     & \hbox{$  t^{-\gamma_\kappa-\frac{d+1}{2}}e^{-\frac{b}{t}\min_{g\in G}|g.x-y|^2 }$;\;  if $0<t<1$.}\\
& \hbox{$e^{-(2\gamma_\kappa+d)t}e^{-b\min_{g\in  G}|g.x-y|^2 }$;\; if $t\geq1$.}
 \end{array}
\right.
$$
But since the function $f$ has compact support and  $g.x\notin supp(f)$ for all $g\in G$ which is finite group, then we can find a ball $B_x$ centered at $x$ and $\rho>0$ such that for all $z\in B_x$, $g\in G$ and $y\in supp(f)$ we have that
$$|g.z-y|^2\geq \rho$$
 and 
$$\left|\frac{1}{\sqrt{t}}\frac{\partial k_t}{\partial x_j} (z,y)\right|\leq \omega(t)=
\left\{
  \begin{array}{ll}
     & \hbox{$  t^{-\gamma_\kappa-\frac{d+2}{2}}e^{-\frac{b}{t}\rho}$;\;  if $0<t<1$,}\\
& \hbox{$e^{-(2\gamma_\kappa+d+\frac{1}{2})t}e^{-b\rho}$;\; if $t\geq1$.}
 \end{array}
\right.
$$
Clearly $\omega$ is integrable on $]0,+\infty[$, which  allows us to differentiate  the integral (\ref{L}) with respect to variable $x_j$ and one has
 $$\delta_j^\kappa L_\kappa^{-\frac{1}{2}}(f)(x)= \frac{1}{\sqrt{\pi}}\int_{\mathbb{R}^d}\int_0^{+\infty}\delta_j^\kappa k_t(x,y)f(y)d\mu_\kappa(y)\frac{dt}{\sqrt{t}}.$$
Now to conclude the theorem  we proceed a follows: Let $(f_n)_n$ be a sequence of functions in $\mathcal{D}$ which converges to $f$ in $L^2(\mathbb{R}^d,d\mu_\kappa)$
and $\varphi$ be a $C^\infty$- function with compact support such that $G.x \cap supp(f)=\emptyset$ for all $x\in supp(\varphi)$ . By continuity of the operators $L_\kappa^{-\frac{1}{2}}$ and
$R_j^\kappa$ on $L^2(\mathbb{R}^d,d\mu_\kappa)$ with the use of (\ref{Du}) and (\ref{def}) we get that
\begin{eqnarray*}
\langle\delta_j^\kappa L_\kappa^{-\frac{1}{2}}(f),\varphi \rangle_\kappa&=&\langle L_\kappa^{-\frac{1}{2}}(f),\delta_j^{\kappa^*}\varphi \rangle_\kappa\\
&=& \lim \langle L_\kappa^{-\frac{1}{2}}(f_n),\delta_j^{\kappa^*}\varphi\rangle_\kappa\\
&=&\lim \langle\delta_j^{\kappa}L_\kappa^{-\frac{1}{2}}(f_n),\varphi \rangle_\kappa\\
&=&\lim \langle R_j^\kappa(f_n),\varphi\rangle_\kappa\\
&=&\langle R_j^\kappa(f),\varphi \rangle_\kappa.
\end{eqnarray*}
Since $\varphi$  is arbitrary then we have  $ R_j^\kappa(f)(x)=\delta_j^\kappa L_\kappa^{-\frac{1}{2}}(f)(x)$ for all $x\in \mathbb{R}^d$ with $G.x \cap supp(f)=\emptyset$ which proves (\ref{io}).
\end{proof}
\begin{prop}\label{prop2}
There exists
 $C>0$ such that for all $y,y_0\in\mathbb{R}^d$,
\begin{eqnarray}\label{K1}
\int_{\min_{g\in G}|g.x-y|>2|y-y_0|}\Big|K_j(x,y)-K_j(x,y_0)\Big| d\mu_\kappa(x)\leq C
\end{eqnarray}
and
\begin{eqnarray} \label{K2}
\int_{\min_{g\in G}|g.x-y|>2|y-y_0|}\Big|K_j(y,x)-K_j( y_0,x)\Big| d\mu_\kappa(x)\leq C.
\end{eqnarray}
\end{prop}
\begin{proof}
We will only show (\ref{K1}), the same  argument can be used  to  prove (\ref{K2}).
Let
\begin{eqnarray*}
h_t(x,y)&=&k_t(x,y)\Big((1-\coth 2t )x_j+ \frac{1}{\sinh 2t }y_j\Big)\\&=&
k_t(x,y)\Big(  (1-\coth 2t) (x_j-y_j)+(1-\tanh t)y_j\Big),\quad x,y\in \mathbb{R}^d,\; t\in]0,+\infty[.
\end{eqnarray*}
 In  view of (\ref{k_j}) we may write
\begin{eqnarray*}
 K_j(x,y)&=&\frac{1}{\sqrt{\pi}}\int_0^{1}h_t(x,y)\frac{dt}{\sqrt{t}}+\frac{1}{\sqrt{\pi}}\int_1^{+\infty}h_t(x,y)\frac{dt}{\sqrt{t}}.
\\&=& K_j^{(1)}(x,y)+K_j^{(2)}(x,y)
\end{eqnarray*}
 where $x, y\in \mathbb{R}^d$, $y\notin G.x$. We claim that $K_j^{(1)}$ and $K_j^{(2)}$ satisfy (\ref{K1}).
Making use of Lemma \ref{lem1}-$(vi)$, (\ref{p}) and  (\ref{int}) we have that
\begin{eqnarray*}
 \int_{\mathbb{R}^d}|K_j^{(2)}(x,y)|d\mu_\kappa(x)&\leq&
 C \int_{\mathbb{R}^d}\int_{1}^{+\infty}e^{-(2\gamma_\kappa+d)t}\tau^\kappa_{-y}(e^{-b|.|^2})(x)
\frac{dt}{\sqrt{t}}d\mu_\kappa(x) \\
&\leq&C \int_{1}^{+\infty}\int_{\mathbb{R}^d}e^{-(2\gamma_\kappa+d)t} e^{-b|z|^2}
d\mu_\kappa(z)\frac{dt}{\sqrt{t}}\leq C.
\end{eqnarray*}
Thus we get
\begin{eqnarray*}
 \int_{\min_{g\in G}|g.x-y|>2|y-y_0|}\Big| K_j^{(2)}(x,y)-K_j^{(2)}(x,y_0)\Big|d\mu_\kappa(x)\leq 2  \int_{\mathbb{R}^d}|K_j^{(2)}(x,y)d\mu_\kappa(x)\leq C.
\end{eqnarray*}
 In order to establish (\ref{K1}) for $K_j^2$ we need to estimate $\dfrac{\partial h_t}{\partial y_i}(x,y )$ for $0<t\leq1$,  which can easily be deduced from
Lemmas \ref{lem1}-$(iv)$ and  \ref{lem3}-$(ii)$. In fact,
$$\frac{\partial h_t}{\partial y_i}(x,y)=\left\{
    \begin{array}{ll}
       & \hbox{$\frac{\partial k_t}{\partial y_i}(x,y)\Big(  (1-\coth 2t) (x_j-y_j)+(1-\tanh t)y_j\Big)$,\; if $i\neq j$} \\
       & \hbox{$ \frac{\partial k_t}{\partial y_j}(x,y)\Big(  (1-\coth 2t) (x_j-y_j)+(1-\tanh t)y_j\Big)+\frac{1}{\sinh2t}k_t(x,y)$}
    \end{array}
  \right.
$$
 and then we obtain for some constant $c>0$,
$$\Big|\frac{\partial h_t}{\partial y_i}(x,y)\Big|\leq Ct^{-\gamma_\kappa-\frac{d }{2}-1}\;\sum_{\alpha\in \mathcal{R}_+}\tau^\kappa_{\sigma_\alpha.x}(e^{-\frac{c}{ t}|.|^2})(-y),\quad 0<t\leq1. $$
Now by mean  value  theorem,
\begin{eqnarray*}
\Big| K_j^{(1)}(x,y)-K_j^{(1)}(x,y_0)\Big|&=&  \frac{1}{\sqrt{\pi}} \int_{0}^{1}|h_t(x,y)-h_t(x,y_0)|\frac{dt}{\sqrt{t}}  \\
 &\leq & \frac{1}{\sqrt{\pi}}|y-y_0|\int_{0}^{1}\int_{0}^{1}\sum_{i=1}^d\Big|\frac{\partial h_t}{\partial y_i}(x,y_\theta)\Big|d\theta\frac{dt}{\sqrt{t}}\\
&\leq & C|y-y_0|\sum_{\alpha\in \mathcal{R}_+}\int_{0}^{1}\int_{0}^{1}t^{-\gamma_\kappa-\frac{d }{2}-\frac{3}{2}}\;\tau^\kappa_{\sigma_\alpha.x}(e^{-\frac{c}{ t}|.|^2})(-y_\theta)d\theta dt
\end{eqnarray*}
where $y_\theta=y_0+\theta(y-y_0)$. Observe that when   $\min_{g\in G}|g.x-y|>2|y-y_0|$ we have
  $$\min_{g\in G}|g.x-y_{\theta}|\geq \min_{g\in G}|g.x-y|- |y-y_\theta|>|y-y_0|.$$
This is an important fact, since  from
(\ref{trad}) and (\ref{sq}) we can write
$$\tau^\kappa_{\sigma_\alpha.x}(e^{-\frac{c}{ t}|.|^2})(-y_\theta )\leq  \tau^\kappa_{\sigma_\alpha.x}\Big(e^{-\frac{c}{ 4t}(|.|+|y-y_0|)^2}\Big)(-y_\theta ), \qquad\text{for all}\;\alpha\in \mathcal{R}_+.$$
Hence,  using (\ref{p}) and (\ref{int}) we get  \\\\
$\displaystyle{\int_{\min_{g\in G}|g.x-y|>2|y-y_0|}\Big| K_j^{(1)}(x,y)-K_j^{(1)}(x,y_0)\Big|d\mu_\kappa(x) }$\\
\begin{eqnarray*}
 &\leq & C|y-y_0|\sum_{\alpha\in \mathcal{R}_+}\int_{0}^{1}\int_{0}^{1}t^{-\gamma_\kappa-\frac{d }{2}-\frac{3}{2}}
\left(\int_{\mathbb{R}^d}\;\ \tau^\kappa_{-y_\theta}\Big(e^{-\frac{c}{ 4t}(|.| +|y-y_0|)^2}\Big)( \sigma_\alpha.x)
d\mu_\kappa(x)\right) \;dt\;d\theta
 \\&\leq&C|y-y_0|\int_{0}^{1}t^{-\gamma_\kappa-\frac{d }{2}-\frac{3}{2}}\int_{\mathbb{R}^d}e^{-\frac{c}{ 4t}(|z| +|y-y_0|)^2}d\mu_\kappa(z)\;dt\\
 &\leq &C|y-y_0|\int_{0}^{+\infty}r^{2\gamma_\kappa+d-1}\left(\int_{0}^{1}t^{-\gamma_\kappa-\frac{d }{2}-\frac{3}{2}}e^{-\frac{c }{4 t}(r+|y-y_0|)^2 } dt\right)\;dr
 \\ &\leq&C|y-y_0|\int_{0}^{+\infty}\frac{r^{2\gamma_\kappa+d-1}}{(r+|y-y_0|)^{2\gamma_\kappa+d+1}}\;dr\int_{0}^{+\infty}u^{-\gamma_\kappa-\frac{d }{2}-\frac{3}{2}}e^{-\frac{c}{ 4u} }\;du\\&\leq &C|y-y_0|\int_{0}^{+\infty}\frac{dr}{(r+|y-y_0|)^{2}}= C.
 \end{eqnarray*}
This finishes the proof of Proposition \ref{prop2} and concluded Theorem \ref{th2}.
\end{proof}


\begin{thebibliography}{}
\bibitem{B}
 B. Amri and  M. Sifi
\textit{ Riesz transforms for Dunkl transform},
Annales mathématiques Blaise Pascal, \textbf{19} (2012),247-262.
\bibitem{BM}
 B. Amri and  M. Sifi
\textit{ Singular integral operators in Dunkl setting}, J. Lie Theory,\textbf{22}(2012), 723-739.
\bibitem{BA} B. Amri, A. Gasmi and M. Sifi,
\textit{Linear and Bilinear Multiplier Operators for the Dunkl Transform
}, Mediterr. J. Math.\textbf{7}  (2010), 503-521.
\bibitem{FW} F. Dai and H. Wang, \textit{A transference theorem for the Dunkl transform and its applications},
J. Funct. Anal.
 \textbf{258} (2010), 4052-4074.
\bibitem{Du}
C. F. Dunkl,
\textit{Differential-difference operators associated to reflection groups},
Trans. Amer. Math. Soc.\textbf{311} (1989), 167-183.
\bibitem{CD}
C. F. Dunkl,
\textit{Integral kernels with reflection group invariance}, Canadian J. Math. \textbf{43} (1991), 1213-1227.
\bibitem{Je}
M.F.E. de Jeu,
\textit{The Dunkl transform},
Invent. Math. \textbf{113} (1993),  147-162.
\bibitem{NS}
A. Nowak and K. Stempak, \textit{\textit{Riesz transform for the Dunkl harmonic oscillator}}, Math. Z. \textbf{262} (2009), 539-556.
\bibitem{R1}
M. R\"osler,
\textit{Generalized Hermite polynomials and the heat equation for the Dunkl operators},
Comm. Math. Phys. \textbf{192} (1998), 519-542.
\bibitem{R4}
 M. R\"{o}sler, \textit{Positivity of Dunkl's intertwining operator }, Duke Math. J. \textbf{98} (1999), 445-463.
\bibitem{R2}
M. R\"osler,
\textit{Dunkl operators\,: Theory and applications\/},
in \textit{Orthogonal polynomials and special functions
(Leuven, 2002)\/},
Lect. Notes Math. 1817, Springer-Verlag (2003), 93-135.
\bibitem{R3}
M. R\"osler,
\textit{A positive radial product formula for the Dunkl kernel},
Trans. Amer. Math. Soc. \textbf{355} (2003), 2413-2438.
\bibitem{St}
 E. M. Stein, \textit{Harmonic Analysis: Real-Variable Methods, Orthogonality and Oscillatory Integrals}, Princeton, New Jersey
 1993.
\bibitem{Stemp}
K. Stempak and J. L. Torrea, \textit{BMO results for operators associated to Hermite expansions},
 Illinois J. Math. Volume \textbf{49} (2005), 1111-1131.
\bibitem{Th}
S. Thangavelu, \textit{On conjugate Poisson integrals and   Riesz transforms for the Hermite expansions}, Colloq. Math. \textbf{64} (1993), 103-113.
 \bibitem{Xu}
S. Thangavelu and Y. Xu, \textit{Convolution operator and maximal function for Dunkl transform,
} J. Anal. Math. \textbf{97} (2005) 25-55.



\end{thebibliography}
\end{document}